\documentclass[12pt,letterpaper]{amsart}
\usepackage{amsthm,amsfonts,amssymb,amsmath}
\usepackage{amsmath, bm}
\usepackage{tabularx}
\usepackage{graphicx}
\usepackage{tikz}
\usetikzlibrary{patterns}
\usetikzlibrary{arrows.meta}
\usetikzlibrary{arrows}
\usepackage{amscd}
\usepackage{parskip}
\usepackage{enumitem}
\usepackage{comment}  
\usepackage{cancel} 
\usepackage{mathtools} 
\usepackage{cleveref} 
\usepackage[enableskew,vcentermath]{youngtab}
\usepackage{ytableau}
\usepackage{tcolorbox}
\usepackage[all]{xy} 
\usepackage{multicol}

\newcommand{\scr}{\scriptstyle}

\newcommand{\bK}{\mathbb{K}}

\newtheorem{theorem}{Theorem}[section]
\newtheorem{lemma}[theorem]{Lemma}
\newtheorem{corollary}[theorem]{Corollary}

\newtheorem{proposition}[theorem]{Proposition}

\theoremstyle{definition}
\newtheorem{definition}[theorem]{Definition}
\newtheorem{example}[theorem]{Example}

\newtheorem{remark}[theorem]{Remark}

\newlength{\cellsize}
\newcommand\tableau[1]{
\vcenter{
\let\\=\cr
\baselineskip=-16000pt
\lineskiplimit=16000pt
\lineskip=0pt
\halign{&\tableaucell{##}\cr#1\crcr}}}

\newcommand{\tableaucell}[1]{{%
\def \arg{#1}\def \void{}%
\ifx \void \arg
\vbox to \cellsize{\vfil \hrule width \cellsize height 0pt}%
\else
\unitlength=\cellsize
\begin{picture}(1,1)
\put(0,0){\makebox(1,1)[c]{$#1$}}
\put(0,0){\line(1,0){1}}
\put(0,1){\line(1,0){1}}
\put(0,0){\line(0,1){1}}
\put(1,0){\line(0,1){1}}
\end{picture}%
\fi}}

\DeclareMathOperator{\comp}{comp}

\DeclareMathOperator{\set}{set}

\newcommand{\rdI}{\mathcal{R}\mathfrak{S}^*}

\DeclareMathOperator{\rw}{rw}

\DeclareMathOperator{\inv}{inv}

\newcommand{\SIT}{\ensuremath{\operatorname{SIT}}}
\newcommand{\SET}{\ensuremath{\operatorname{SET}}} 

\newcommand{\hn}{H_n(0)}

\newcommand{\ninv}{\mathrm{inv}} 
\newcommand{\row}{\mathrm{row}} 
\newcommand{\col}{\mathrm{col}} 

\newcommand{\suchthat}{\;|\;}

\newcommand{\spam}{\operatorname{span}}

\DeclareFontFamily{U}{rcjhbltx}{}  
\DeclareFontShape{U}{rcjhbltx}{m}{n}{<->rcjhbltx}{}
\DeclareSymbolFont{hebrewletters}{U}{rcjhbltx}{m}{n}
\DeclareMathSymbol{\shin}{\mathord}{hebrewletters}{152}

\usetikzlibrary{positioning, arrows.meta}

\author[Lafreni\`ere, Orellana,  Pun, Sundaram, van Willigenburg, Whitehead McGinley]{Nadia Lafreni\`ere,  Rosa Orellana, Anna Pun, Sheila Sundaram, Stephanie van Willigenburg and Tamsen Whitehead McGinley}

\address{Nadia Lafreni\`ere: Department of Mathematics and Statistics, Concordia University, Montr\'{e}al, QC, Canada} 
\email{nadia.lafreniere@concordia.ca}
\address{Rosa Orellana: Department of Mathematics, Dartmouth College, Hanover, NH, USA}
\email{rosa.c.orellana@dartmouth.edu}
\address{Anna Pun: Department of Mathematics, Baruch College, New York, NY, USA}
\email{anna.pun@baruch.cuny.edu}
\address{Sheila Sundaram: School of Mathematics, University of Minnesota, Minneapolis, MN, USA}
\email{shsund@umn.edu}
\address{Stephanie van Willigenburg: Department of Mathematics, University of British Columbia, Vancouver, BC, Canada}
\email{steph@math.ubc.ca}
\address{Tamsen Whitehead {McGinley:} Department of Mathematics, Santa Clara University, Santa Clara, CA, USA}
\email{tmcginley@scu.edu}

\title[The skew extended Hecke poset]{Minimal elements in the  skew extended 0-Hecke poset}
\date{\today}

\begin{document}


\begin{abstract}{The row-strict 0-Hecke action on standard immaculate skew  tableaux was studied by the present authors, who showed that it gives rise to a bounded poset, called the \emph{skew immaculate Hecke poset}, and consequently to a cyclic 0-Hecke module. It was further shown that the subposet of skew standard extended immaculate  tableaux always has a unique maximal element, but may have multiple minimal elements.  In this paper we focus on these minimal elements, completely classifying them for a family of skew shapes that we call \emph{lobsters}. Moreover, we prove that when the skew shape is connected, the skew extended Hecke poset does have a unique minimal element, thereby showing that the associated 0-Hecke module is cyclic for both the row-strict and the dual immaculate actions. }
\end{abstract}

\maketitle

\section{Introduction}


The papers \cite{NSvWVW2023, NSvWVW2024}  studied modules of the $0$-Hecke algebra arising from actions on several sets of standard tableaux of composition shape.  These results were extended to actions  on skew-shaped tableaux by the authors of this paper \cite{LOPSvWwM2025}. In \cite{LOPSvWwM2025,NSvWVW2023, NSvWVW2024}, modules of the $0$-Hecke algebra
were defined via actions on sets of tableaux. These actions give the sets of tableaux a poset structure whose combinatorics can be used to describe representation-theoretic properties of the $0$-Hecke modules such as indecomposability or cyclicity. In the present work, we continue our study of the combinatorics of the   skew 0-Hecke poset, focusing on the one  resulting from the \emph{row-strict immaculate action} on skew standard extended tableaux.


We have endeavored to make this paper completely self-contained, and accessible to a broad audience.  In particular, although the questions we address originate in the representation-theoretic aspects of the $0$-Hecke algebra, no prior knowledge of the latter is required. We now give a more detailed description of our results. Let $\alpha$ and $\beta$ be compositions with $\beta\subseteq \alpha$.
The primary combinatorial objects of interest in this paper are:


\begin{itemize}
\item 
  The \emph{skew composition} $\alpha/\beta$, whose diagram consists of those cells in the diagram of $\alpha$ that are not in the diagram of $\beta$ when placed in the south-west corner of $\alpha$.  The number of cells in  $\alpha/\beta$ is denoted by $|\alpha/\beta|$.
  \item   \emph{Standard extended tableaux} of skew shape $\alpha/\beta$; these are bijective fillings, with the integers $\{1,2,\ldots, |\alpha/\beta|\}$, of the diagram of the skew shape such that all columns increase strictly bottom to top, and all  rows increase strictly left to right. 
\end{itemize}
Detailed definitions, as well as background on the 0-Hecke algebra, appear in \Cref{sec:bgrnd}.  Here we illustrate the  combinatorial objects described above with one example.
Consider $\alpha=(3,2,5,4,1)$ and $\beta=(1,1,3,2)$; these define the skew shape $\alpha/\beta$ pictured below, for which one standard extended filling  is 
\ytableausetup{smalltableaux, boxsize=1.1em}
\begin{center}
\begin{ytableau}
6\\
&&7&9\\
&&&1&5\\
&3\\
&2&4
\end{ytableau}.
\end{center}
\ytableausetup{nosmalltableaux} 

Standard \emph{extended} tableaux form a subset of the standard immaculate tableaux introduced in \cite{BBSSZ2014}.  For standard immaculate tableaux,  the condition on the bijective fillings requires strict increase, left to right, for all rows,  but strict increase, bottom to top, only for the cells remaining in the \emph{first} column of  $\alpha$  after removing $\beta$.  The standard extended tableaux studied in this paper  were first introduced  for composition shape $\alpha$  in \cite{AS2019}, and later studied for skew compositions in \cite{LOPSvWwM2025}.  Both the skew standard immaculate tableaux and the skew standard extended tableaux of shape $\alpha/\beta$ span vector spaces  carrying various actions of the 0-Hecke algebra.

The row-strict 0-Hecke action on standard immaculate tableaux of arbitrary shape was introduced in \cite{NSvWVW2023}. It was shown in \cite{NSvWVW2024} that the row-strict action is related, via an involutive 0-Hecke automorphism,  to the dual immaculate action of the 0-Hecke algebra first defined by Berg, Bergeron, Saliola, Serrano and Zabrocki \cite{BBSSZ2015}.   For composition shape, the paper \cite{NSvWVW2024} studied the 0-Hecke modules resulting from the row-strict action, as well as other variants.   In particular, the \emph{immaculate Hecke poset} for composition shape $\alpha$ was introduced in \cite{NSvWVW2024}, where it was shown to play an important role in determining the structure of the resulting 0-Hecke modules.

We  describe the motivating question for our paper. 
In \cite{LOPSvWwM2025} it was shown that the row-strict 0-Hecke action on skew standard  immaculate  tableaux of shape  $\alpha/\beta$ defines a graded poset, called \emph{the skew immaculate Hecke poset}, that is bounded above and below with rank function determined by the number of inversions of the reading word,  thereby yielding several cyclic  modules for the 0-Hecke algebra.
It was further observed \cite[Proposition 49]{LOPSvWwM2025} that  the subposet of skew standard \emph{extended} tableaux is also  bounded above, but not necessarily bounded below \cite[Example 50 and Figure 1]{LOPSvWwM2025}.  Thus in this case, while the dual immaculate action of \cite{BBSSZ2015} yields a cyclic module generated by the unique top element of the skew poset,  the row-strict 0-Hecke action of \cite{NSvWVW2024} need no longer produce a cyclic module. 

 The present paper investigates the minimal elements of the skew standard extended poset $\SET(\alpha/\beta)$, for a fixed skew shape  $\alpha/\beta$.  These elements form a  minimal set of generators for the 0-Hecke module defined by the row-strict 0-Hecke action of \cite{LOPSvWwM2025} and \cite{NSvWVW2024} on the standard extended tableaux of skew shape  $\alpha/\beta$.
 
 Our first result, \Cref{thm:min_element_connected}, asserts that  for \emph{connected} diagrams of skew compositions, the poset $\SET(\alpha/\beta)$ has a unique minimal element; hence   the associated 0-Hecke module is cyclic.   We then undertake a detailed study of a special class of skew compositions that we call \emph{lobsters}. We will refer to this set of skew compositions as the \emph{lobster} class.  Here we give the size of the whole poset, which is also the dimension of the associated 0-Hecke module.
 
 For the lobster class we completely classify the minimal elements of the skew standard extended  Hecke poset, in \Cref{thm:lobster-min-elts}.  
 An algorithm to generate all the minimal elements  is also described.  We show that there can now be maximal chains of different lengths in the poset.   In the process we uncover some surprisingly  beautiful and elegant combinatorics  underlying the structure of these  minimal elements.  A detailed enumeration of the number of minimal elements and the lengths of the maximal chains is included in \Cref{thm:min_elts}.

 The paper is organised as follows.  \Cref{sec:bgrnd} reviews the necessary background, particularly definitions and results from \cite{LOPSvWwM2025}.  In \Cref{sec:key-lemmas}, we prove a useful lemma characterising minimal elements in $\SET(\alpha/\beta)$. \Cref{sec:connected-skew} analyses the case of connected skew shapes, and \Cref{sec:disconnected-skew} deals with the disconnected case. In this latter section we define the lobster shape 
 and reveal the pleasing features of the (no longer graded) skew standard extended 0-Hecke poset.  
 
\section{Background and definitions}\label{sec:bgrnd}

General references for this section are  \cite{LMvW2013},  \cite{MathasHecke1999}, \cite{RPS-EC1-2012} and \cite{LOPSvWwM2025}.

\subsection{Compositions, inversions and tableaux}

A {\emph{composition}} of  $n$ is a  sequence of positive  integers $\alpha = (\alpha_1, \alpha_2, \ldots,\alpha_k)$ summing to $n$, its \emph{size}. The $\alpha_i$ are the \emph{parts} of $\alpha$. We denote the number of parts,  $k$, by $\ell(\alpha)$, and the size $n$ of the composition by $|\alpha|$.
Write $\alpha\vDash n$ for a composition $\alpha = (\alpha_1, \alpha_2,\ldots,\alpha_k)$ of $n$, {and denote by $\emptyset$ the empty composition of 0}.

Compositions of $n$ are in bijection with subsets of $\{1,2,\ldots,n-1\} =[n-1]$. 
For a composition $\alpha\vDash n$, the corresponding set is $\set(\alpha) = \{\alpha_1,\alpha_1+\alpha_2,\ldots,\alpha_1+\cdots+\alpha_{k-1}\}$. 
Given a subset $S=\{s_1<s_2<\cdots<s_j\}$ of $[n-1]$, the corresponding composition of $n$ is $\comp(S)=(s_1,s_2-s_1,\ldots,s_j-s_{j-1},n-s_j)$. 

The \emph{diagram} of a composition $\alpha$ is a geometric 
depiction of $\alpha$  as left-justified cells, with 
$\alpha_i$ cells in the $i$th row, counting the bottom row as row 1. Given two compositions $\alpha, \beta$, we say that $\beta \subseteq \alpha$ if $\beta _j \leq \alpha _j$ for all $1\leq j \leq \ell (\beta) \leq \ell (\alpha)$, and given $\alpha, \beta$ such that $\beta \subseteq \alpha$, the \emph {skew diagram}  $\alpha / \beta$ is the array of {cells} in $\alpha$ but not $\beta$ when $\beta$ is placed in the bottom-left corner of $\alpha$.   For example, the diagrams of $\alpha={(4,2,3)}$   and $\alpha/\beta$ where $\beta = (2,1)$ are  respectively
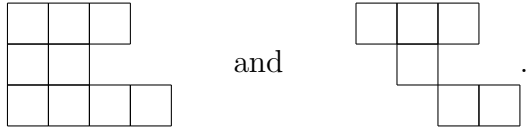
\[\tableau{\phantom{} &\phantom{} &\phantom{} \\\phantom{} &\phantom{}\\ \phantom{} &\textbf{} &\phantom{} & \phantom{}} \qquad \text{and } \qquad {\tableau{\phantom{} &\phantom{} &\phantom{} \\  &\phantom{}\\   &  &\phantom{} & \phantom{}} }.\]
If $\beta=\emptyset$, then $\alpha/\beta=\alpha$.

We will usually suppress the cells of $\beta $ in the skew-diagram $\alpha/\beta$. However, sometimes, for clarity, it will be necessary to draw the cells of $\beta$, for example, when skewing by $\beta$ removes an entire row or column of $\alpha$.  

\begin{definition}\label{def:skew-SIT} 
Let $\alpha,\beta$ be compositions with $\beta \subseteq \alpha$. A \emph{skew standard immaculate tableau} of {\emph{shape}} $\alpha/\beta$ is a filling $T$ of the cells of the skew diagram $\alpha/\beta $ with distinct entries $1, 2, \ldots , |\alpha/\beta|=|\alpha|-|\beta|$ such that 
\begin{enumerate}[itemsep=1pt]
\item the leftmost (possibly empty) column  entries belonging to $\alpha$ but not $\beta$  
increase from bottom to top;
\item the entries in each row  increase from left to right. 
\end{enumerate}
In particular, if $\ell(\alpha)=\ell(\beta)$, then there is no restriction on the columns.
Let $\SIT(\alpha/\beta)$ denote the set {of}  skew standard  immaculate tableaux of shape $\alpha/\beta$.  
\end{definition}
\begin{example}\label{ex:skewSIT}  Let $\alpha=(3,4,4,1)$, $\beta=(1,2)$. Then 
 \ytableausetup{smalltableaux, boxsize=1.1em}
    $T=\,\,$\begin{ytableau}
        4 \\ 2  & 3 &6 &8\\\ &  \  &1 & 7\\ \ 
        & 5 &9 
    \end{ytableau}
    \ytableausetup{nosmalltableaux}
 is in $\SIT(\alpha/\beta)$.   
 \end{example}

\begin{definition}\label{def:skew-SET}  Let $\alpha,\beta$ be compositions with $\beta \subseteq \alpha$. 
Define $\SET(\alpha/\beta)$ to be the subset of $\SIT({\alpha/\beta})$ consisting of all \emph{skew standard extended tableaux} of \emph{shape} ${\alpha/\beta}$, that is, all tableaux $T$ of shape ${\alpha/\beta}$ where {all column entries increase bottom to top.}\end{definition}
\begin{definition}\label{def:skewels} For $\alpha \vDash n, \beta\vDash m$ with $\beta\subseteq\alpha$, we define the following special standard tableaux in $\SET(\alpha/\beta)$.
\begin{itemize}
\item $S^{\row}_{\alpha/\beta}$ is the \emph{row superstandard}  tableau  whose rows are filled left to right, \emph{bottom to top},  beginning with the bottom row and moving up, using the numbers $1,2, \ldots, n-m$ taken in consecutive order.  
\item $S^{\col}_{\alpha/\beta}$ is the \emph{column superstandard}   tableau  whose columns are  filled   bottom to top,  left to right, beginning with the leftmost column and moving   right, using the numbers $1,2,\ldots, n-m$  taken in consecutive order.  
 \end{itemize}
\end{definition}

\begin{example}\label{ex:special-skew-SIT-1}  Let $\alpha=(2,2,3,2,4)  \vDash 13$, $\beta=(2,1,2)\vDash 5$.  
Then 
\begin{center}
\ytableausetup{smalltableaux, boxsize=1.1em}
    $S^{\row}_{\alpha/\beta}=$\ \begin{ytableau}
        5& 6&7 &8\\ 3 & 4\\\  &\ &2\\ \  &1\\\  &\ 
    \end{ytableau}, \quad 
    $S^{\col}_{\alpha/\beta}=$ \ \begin{ytableau}
        2& 5&7 &8\\ 1 & 4\\ \  &\ &6\\ \   &3\\\  &\ 
    \end{ytableau}.
    \ytableausetup{nosmalltableaux}
    \end{center}
\end{example}

\begin{example}\label{ex:special-skew-SIT-2}  Let $\alpha=(5,4,6)\vDash 15$, $\beta=(2,1,2)\vDash 5$. Note that here $\ell(\alpha)=\ell(\beta)=3$.  

\begin{center}
\ytableausetup{smalltableaux,boxsize=1.1em}
    $S^{\row}_{\alpha/\beta}=$ \begin{ytableau}
        \  &\  &7 &8 &9 &10 \\ \    & 4&5 & 6 \\ \  &\  &1 &2 &3
    \end{ytableau}, \quad 
    $S^{\col}_{\alpha/\beta}=$  \begin{ytableau}
        \  & \  &4 &7 &9 &10 \\ \   &1 & 3 &6\\ \  &\  &2 &5 &8 
    \end{ytableau}.
    \ytableausetup{nosmalltableaux}
    \end{center}
\end{example}
Recall (e.g., \cite{RPS-EC1-2012}) that the \emph{inversion set of a permutation $\sigma \in S_n$} is
$$\mathrm{Inv} (\sigma) = \{ (p,q)\suchthat 1\leq p < q \leq n \mbox{ and } \sigma(p) > \sigma (q)\}.$$  The \emph{number of inversions} of $\sigma$  is denoted by $\ninv(\sigma) = |\mathrm{Inv}(\sigma)|.$ 

The \emph{reading word} $\rw(T)$ of a  tableau $T$ of skew shape $\alpha/\beta$ is  the word obtained by reading the entries of $T$  from right to left along rows, and from top to bottom. In Example \ref{ex:special-skew-SIT-1}, the reading word of $S^{\col}_{\alpha/\beta}$ is $87524163$, with 21 inversions. 
We can therefore define the number of inversions of a standard immaculate tableau $T $ to be the number of inversions in its reading word $\rw(T)$, and simply write $\inv(T)$ for $\inv \left(\rw(T)\right)$.   In particular, one sees that for any skew composition $\alpha/\beta$, $\inv(S^{\row}_{\alpha/\beta})=\binom{|\alpha/\beta|}{2}.$

\subsection{The 0-Hecke algebra}\label{sec:0-Hecke-algebra}
\begin{definition}\cite{MathasHecke1999} Let $\bK$ be any field.  The {\emph{0-Hecke algebra}} $H_n(0)$ is the $\bK$-algebra with 
generators $\pi_i, 1\le i\le n-1$, $n\ge 2$, and relations 
\begin{center}$ {\pi_i}^2 =\pi_i; \quad 
          \pi_i\pi_{i+1}\pi_i =\pi_{i+1}\pi_i\pi_{i+1};\quad
              \pi_i\pi_j =\pi_j\pi_i, \ |i-j|\ge 2.$\end{center}\end{definition}
When $n=1$ we have the trivial case $H_1(0)=\bK$.
 
The algebra $H_n(0)$ has dimension $n!$ over  $\bK$, with basis elements $\{\pi_\sigma: \sigma\in S_n\},$ where $\pi_\sigma=\pi_{i_1}\cdots \pi_{i_m}$ if $\sigma=s_{i_1}\cdots s_{i_m}$  is a reduced word.  This is well defined by standard Coxeter group theory, see \cite{MathasHecke1999}.

It is known  \cite{PamelaBromwichNorton1979} that the 0-Hecke algebra {$H_n(0)$} admits  precisely $2^{n-1}$  simple modules $L_\alpha$, 
one for each composition $\alpha\vDash n$, and all one-dimensional. Recall that $\set(\alpha)$ is the subset of $[n-1]$ associated to the composition $\alpha$.  The 0-Hecke action on the module $L_\alpha = \spam\{v_\alpha\}$ is defined as follows.  For each $i=1, \ldots, n-1$, we have 
$\pi_i(v_\alpha)=\begin{cases} 0 , &\text{if }  i\in \set(\alpha),\\
                                                v_\alpha, &\text{otherwise}.\end{cases}$

For the remainder of this paper we take the ground field to be the field  $\mathbb{C}$ of complex numbers.   To avoid trivialities we will also always assume that $n\ge 2$ when referring to the $\hn$-modules.

\subsection{The row-strict dual immaculate action}
\label{sec:rs-action} 
 We begin by recalling from \cite{NSvWVW2024, LOPSvWwM2025} the definition of  the row-strict immaculate action, referred to as the $\rdI$-action, which gives the vector space with basis  $\SIT(\alpha/\beta)$ 
 of skew standard immaculate tableaux,   the structure of a  module for the 0-Hecke algebra. Let $s_i$ be the operator switching the entries $i$, $i+1$ in a tableau $T$. The $\rdI$-action of the  Hecke generator $\pi_i$  on  $T\in \SIT(\alpha/\beta)$ is given as follows.  
\[\pi_i^{\rdI}(T)=\pi _i(T)=\begin{cases} T, &  \text{if $i+1$ is strictly above $i$ in $T$},\\
                        s_i(T), & \text{if $i+1$ is strictly below $i$ in $T$},\\
                        0, & \text{otherwise}.
 \end{cases}\]

Extending the results of \cite[Sections 5 and 6]{NSvWVW2024}, it was shown in \cite[Theorem 41]{LOPSvWwM2025} that this action has the following features.
\begin{enumerate}
\item  It defines a bounded and graded poset $P\rdI_{\alpha/\beta}$ on the elements of $\SIT(\alpha/\beta)$. 
The number of inversions in the reading word $\rw(T)$ of $T$  determines the rank of a standard immaculate tableau $T\in \SIT(\alpha/\beta)$; for the precise statements see \cite[Remark 6.14]{NSvWVW2024} and \cite[Proposition 48]{LOPSvWwM2025}.

\item It makes the vector space with basis $\SIT(\alpha/\beta)$ a cyclic module that we call $\mathcal{V}_{\alpha/\beta}$, for the 0-Hecke algebra, generated by the  unique minimal element in the poset.
\end{enumerate}

Now we turn to the central objects in this paper, the skew standard extended tableaux.  It was shown in  \cite[Proposition 29]{LOPSvWwM2025} that the subspace of $\SIT(\alpha/\beta)$ with basis $\SET(\alpha/\beta)$
is itself a 0-Hecke module for  the  $\rdI$-action.  Recall (e.g. \cite{RPS-EC1-2012}) that the length of a chain $x_0<x_1<\cdots<x_r$ in a poset is $r$,  one less than the number of elements in the chain.
\begin{theorem}\cite[Proposition 49]{LOPSvWwM2025}\label{prop:skew-SET-module-top-element} Let $\alpha ,\beta $ be compositions with $\beta \subseteq \alpha$. Then the subposet of $P\rdI_{\alpha/\beta}$ whose elements are in  $\SET(\alpha/\beta)$ has a unique maximal element, namely the unique maximal element $S^{\row}_{\alpha/\beta}$ of $P\rdI_{\alpha/\beta}$. Moreover:

\begin{enumerate}
\item 
This  subposet has the same rank function as $P\rdI_{\alpha/\beta}$, meaning that for fixed 
$T\in \SET(\alpha/\beta)$, all maximal chains from $T$ to the maximal element $S^{\row}_{\alpha/\beta}$ have the same length, equal to $\inv(S^{\row}_{\alpha/\beta})-\inv(T)$.   
\item 
 For the $\rdI$-action, 
the  minimal elements of the subposet   generate 
 $\mathrm{span}\{T :T\in\SET(\alpha/\beta)\}$ as a submodule of  the cyclic $\rdI$-module  $\mathcal{V}_{\alpha/\beta}$ defined on the set $\SIT(\alpha/\beta)$.
\end{enumerate}
\end{theorem}

The preceding theorem is the generalisation of the following result of \cite{NSvWVW2024} to skew tableaux.

\begin{theorem}[{\cite[Lemma 7.9, Theorem 7.13]{NSvWVW2024}}]\label{thm:cyclicSET-alpha}
 The subposet $\SET(\alpha)$ of $\SIT(\alpha)$ coincides with the interval $[S^{\col}_{\alpha}, S^{\row}_\alpha] $ of $\SIT(\alpha)$. Hence the 0-Hecke module with basis  $\SET(\alpha)$ 
 is an $\rdI$-submodule of the module $\mathcal{V}_{\alpha}$ with basis $\SIT(\alpha)$, and it is cyclically generated by $S^{\col}_{\alpha}$.  It is indecomposable.
\end{theorem}

By abuse of notation, we will use $\SET(\alpha/\beta)$ to refer to both the 0-Hecke $\rdI$-poset structure and   the set of standard extended tableaux, relying on context to make the distinction clear.  As mentioned in the Introduction, our goal is to study the structure of the Hecke poset $\SET(\alpha/\beta)$, focusing in particular on the minimal elements.

We emphasise here that since the Hecke poset $\SET(\alpha/\beta)$ need not have a unique minimal element, it is not a bounded or  graded 
poset in the customary sense \cite{RPS-EC1-2012}. 
\Cref{fig:ungradedSET} shows an example where maximal chains have different lengths. 
Here we have $\alpha=(4,2,4)$ and $\beta=(2,1,3)$. 
 
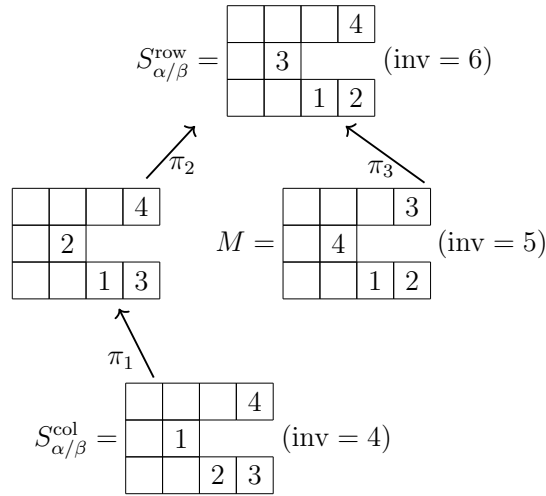
\begin{figure}[htb] \centering
\begin{center}{
		\scalebox{.9}{			
			\begin{tikzpicture}
			\newcommand*{\xdist}{*2}
			\newcommand*{\ydist}{*1.8}
\node (n3-11) at (.35\xdist, 2.4\ydist)  {$(\inv=4)$};   
\node (n31) at (-1\xdist,2.4\ydist) 
			{$S^{\col}_{\alpha/\beta}= 
			\tableau{ 
				\  &\ &\   &4  \\ 
				\  &1 \\
                   \      &\  &2 &3}  $
                                        } ;
\node (n71) at (-1.5\xdist,4\ydist) 
			{$\tableau{ 
			\	  &\ &\   &4  \\ 
			\	  &2 \\
                                     \   &\ & 1 &3}  $
                                        } ;
  \node (n7-12) at (1.5\xdist, 4\ydist)  {$(\inv=5)$};
\node (n72) at (.25\xdist,4\ydist) 
			{$M= \tableau{ 
				\  &\ &\   &3  \\ 
				\  &4 \\
                                     \   &\ & 1 &2}  $
                                        } ;
 \node (n8-11) at (1.1\xdist, 5.5\ydist)  {$(\inv=6)$};                                    
\node (n8) at (-0.25\xdist,5.5\ydist) 
			{  ${S^{\row}_{\alpha/\beta} = 
				\tableau{ 
				\ &\ &\   &  4 \\ 
				\ &3 \\
                                   \    &\  & 1 &2} }  $ 
			}; 
\draw [thick, ->] (-1\xdist,2.9\ydist)  -- (n71) node [near start, left] {$\pi_{1}$}; 

 \draw [thick, ->] (n71) -- (n8) node [near start, right] {$ \pi_{2}$}; 
\draw [thick, ->] (1\xdist,4.5\ydist) -- (n8) node [near start, left] {$\pi_{3}$};               	
			\end{tikzpicture}	
}			
}\end{center}
\caption{The ungraded skew Hecke poset $\SET(\alpha/\beta)$ for $\alpha=(4,2,4), \beta=(2,1,3)$, with two minimal elements $M$ and $S^{\col}_{\alpha/\beta}$ having different numbers of inversions.}\label{fig:ungradedSET}
\end{figure}
\subsection{Classification of skew compositions}
In order to analyse the Hecke poset $\SET(\alpha/\beta)$ associated to a skew composition 
$\alpha/\beta$, we  first classify skew compositions according to the nature of their connectedness, as illustrated in  \Cref{tab:skew_diagram_shapes}.  

The following definition is motivated by the fact that 
the  structure of the skew Hecke poset is unaffected by the presence of empty rows in the skew shape.  This is  a consequence of the more general observation that there are many pairs $(\alpha, \beta)$ giving the same skew shape $\alpha/\beta$.  

\begin{definition}\label{def:reduced_disjoint}
Call $\alpha/\beta$ \emph {reduced} if it has no empty rows (i.e. rows $i$ with $\alpha_i=\beta_i$), as a skew shape.

As an example, let $\alpha = (2,2,3)$ and $\beta=(1,2,2)$. Then $(2,2,3)/(1,2,2)$ can be reduced to  $(2,3)/(1,2)$:
\ytableausetup{smalltableaux,boxsize=1.1em}
$$ (2,2,3)/(1,2,2) = \begin{ytableau}
\bullet&\bullet& \\
\bullet&\bullet \\
\bullet&\end{ytableau},
\text{ which we reduce by removing the middle row: }
\begin{ytableau}
    \bullet&\bullet&\\
    \bullet&\end{ytableau}.$$

Call $\alpha/\beta$  \emph{connected} if, in its reduced shape, for every pair of  consecutive rows, there is at least  one column in which there is a cell in each row of the pair.

Call $\alpha/\beta$ a \emph{disjoint sum} of two  nonempty skew shapes $A$ and  $B$
if $A,B \subseteq \alpha/\beta$, $A\cup B = \alpha/\beta$, $A\cap B = \emptyset$, and the smallest interval containing  all the rows (resp. columns) of $A$ is disjoint from
the smallest interval containing all  the rows (resp. columns) of $B$.

The extension to more than two shapes is clear.  
\end{definition}
    \begin{table}[h!]
        \centering
        \renewcommand{\arraystretch}{1.7}
        \begin{tabular}{c|c|c}
             Connected & Disjoint sum & 
             Neither connected nor a disjoint sum\\
             \hline
             \rule{0pt}{10ex}\young(:~~~,~~,:~~) & \young(:::~~,:::~,~~) &\young(:~,~,:~)\\
             \hline 
             $\alpha = (4,3,5), \beta = (2,1,2)$& $\alpha = (3,5,6), \beta = (1,4,4)$ & $\alpha = (3,2,3), \beta = (2,1,2)$\\[2pt]
        \end{tabular}
        
        \caption{Respectively, a connected skew diagram, a disjoint sum of two connected components, and a diagram that is neither connected nor a disjoint sum of connected components. The latter is not a disjoint sum because the cells in the right column are in rows $1$ and $3$, but the interval containing those rows, namely $[1,3]$, is not disjoint from $\{2\}$, the row for the cell in the left column.
        }
        \label{tab:skew_diagram_shapes}
    \end{table}        

\section{A characterisation of the minimal elements in $\SET(\alpha/\beta)$}\label{sec:key-lemmas}

In this section we collect some preliminary observations for our analysis of the minimal elements in $\SET(\alpha/\beta)$.

\begin{lemma}\label{lem:key-min} Let $M\in \SET(\alpha/\beta)$.  Then $M$ is a minimal element if and only if, for every $i=1,2,\ldots, n-1$, one of the following holds:
\begin{enumerate}
\item $i, i+1$ are consecutive elements  in the same row of $\alpha/\beta$: $\tableau{{\scriptstyle i } & {\scriptstyle i+1}}$\, , 
\item $i, i+1$ are  consecutive elements in the same column of $\alpha/\beta$, i.e., with no cell between them:\quad 
$\tableau{{\scriptstyle i +1}\\  {\scriptstyle i}}$ \quad \quad or
\quad\quad $\tableau{{\scriptstyle i+1}\\ \\ {\scriptstyle i}}$\, ,  
\item $i$ is in a row of $\alpha/\beta$ \emph{higher} than $i+1$: 
 $\tableau{{\scriptstyle i}\\ & &   &\\ &  &  &{\scriptstyle i+1}} $ \text{ or }  
$\tableau{& & & {\scriptstyle i}\\  & \\  {\scriptstyle i+1}}$\, .
\end{enumerate}
\end{lemma}
\begin{proof} This is immediate from the definition of the $\rdI$-action in \Cref{sec:rs-action}; $M$ is \emph{not} minimal if and only if there is a pair $(i, T)$,  $1\le i\le n-1$ and $T\in \SET(\alpha/\beta), T\ne M, $ such that \[\pi^{\rdI}_i(T)=\pi_i(T)=M,\] 
i.e., if and only if 
in $M$, $i$ appears in a lower row than $i+1$, and not in the same column.
\end{proof}
\begin{remark}
The preceding lemma immediately leads to the following characterisation of when $S^\col_{\alpha/\beta}$ is a minimal element of $\SET(\alpha/\beta)$:

    $S^\col_{\alpha/\beta}$ is a minimal element of $\SET(\alpha/\beta)$ if and only if 
    for every pair of consecutive nonempty columns, the highest cell in the left column of the pair is in a row weakly above the lowest cell in the right column.
\end{remark}

We close this section with a definition and an observation.  Let $\alpha/\beta$ be any skew shape.  For each cell $c$ of $\alpha/\beta$, define the \textit{left shadow} of $c$ to be  the number $\mathrm{ls}(c)$ of cells weakly southwest of it, excluding the cell itself.  One sees 
  that $\inv(S^{\col}_{\alpha/\beta})=\sum_{c \in \alpha/\beta} \mathrm{ls}(c)$    for any shape.  

\section{$\SET(\alpha/\beta)$ for connected $\alpha/\beta$}\label{sec:connected-skew}

 It is already known that $S^\col_{\alpha/\beta}$ is the unique minimal element of $\SET{(\alpha/\beta)}$ in the special case when $\beta=\emptyset$ (see Theorem~\ref{thm:cyclicSET-alpha} and \cite{NSvWVW2024}) or when  $\alpha/\beta$ is a rim hook \cite[\S 3.2]{HuangJia2016}, i.e., a skew shape such that for every pair of consecutive rows, there is exactly one column  in which both rows have a cell.
  
In analogy with \Cref{thm:cyclicSET-alpha}, \Cref{thm:min_element_connected} below is a generalisation  to arbitrary skew tableaux.  

\begin{theorem}\label{thm:min_element_connected}
    If $\alpha/\beta$ is connected, then $\SET(\alpha/\beta)$ has a unique minimal element, $S^\col_{\alpha/\beta}$.  Equivalently, the skew extended Hecke poset $\SET(\alpha/\beta)$ coincides with the interval $[S^\col_{\alpha/\beta}, S^\row_{\alpha/\beta}]$ in $\SIT(\alpha/\beta)$, and its   rank is $\binom{|\alpha/\beta|}{2}-\sum_{c \in \alpha/\beta} \mathrm{ls}(c)  $. 
\end{theorem}
\begin{proof}
    The key idea is a straightening algorithm on standard extended tableaux, presented in \cite[Lemma 7.9]{NSvWVW2024}. This algorithm takes any standard extended tableau $T$  of a skew connected shape $\alpha/\beta$ and returns $S^\col_{\alpha/\beta}$ through a sequence of standard extended tableaux with fewer inversions. We recall it here for completeness.  

    \begin{enumerate}
        \item Take tableau $T$. Read the first column from bottom to top, and find the lowest element $i$ in that column that differs from $S^\col_{\alpha/\beta}$, say in cell $c$.
        \item Since $S^\col_{\alpha/\beta}$ has the first column labeled with $[\ell]$, where $\ell$ is the number of cells  in the first column of $\alpha/\beta$, $i-1$ is necessarily in some other column of $T$, so to the right. Now $i-1$ cannot be weakly above $i$, otherwise, because  $\alpha/\beta$ is connected,  the strict increase along rows and up columns would be violated.
        
        Hence $i-1$ appears below $i$, so the values $(i-1,i)$ form an inversion in $T$. Swap the positions of $i$ and $i-1$.
        \item Repeat step (2) with the same cell $c$ (now filled with $i-1$)  until its values in $T$ and $S^\col_{\alpha/\beta}$ agree.
        \item Repeat steps (2) and (3) with all entries in the first column.\\
        At this point, the values in the first column are $[\ell]$, so the values to the right are all larger.
        \item Repeat steps (1) -- (4) with the remaining columns, proceeding  left to right.
    \end{enumerate}
    The result is $S^\col_{\alpha/\beta}$.\\
    Each passage through step (2) decreases the number of inversions by $1$, whereas each passage through step (3) decreases the number of distinct entries between $T$ and $S^\col_{\alpha/\beta}$, ensuring that the algorithm terminates. 
    Hence this algorithm gives a chain in  $\SET(\alpha/\beta)$ (for $\alpha/\beta$ connected) from $S^\col_{\alpha/\beta}$ to $T$, for any choice of   $T\ne S^\col_{\alpha/\beta}$, showing that $S^\col_{\alpha/\beta}$ is the unique minimal element of $\SET(\alpha/\beta)$.

    Since the top element of the skew Hecke poset is $S^{\row}_{\alpha/\beta}$, the stated equivalence follows.
The rank of the Hecke poset $\SET(\alpha/\beta)$ is thus given by
\[ \inv(S^{\row}_{\alpha/\beta}) -\inv(S^{\col}_{\alpha/\beta}) 
=  \binom{|\alpha/\beta|}{2}-\sum_{c \in \alpha/\beta} \mathrm{ls}(c). \qedhere\]
\end{proof}

A similar statement can be made for a more general class of skew tableaux $\alpha/\beta$ when the composition $\beta$ is a partition.  \Cref{ex:skew-not-conn-not-disjoint-sum} below illustrates the structure of the skew shape.

\begin{example}\label{ex:skew-not-conn-not-disjoint-sum}  Let $\alpha=(5, 6, 5, 3, 2)$ and let $\beta$ be the partition $(4,3,3,2)$.  Then 
\[\alpha/\beta=
\ytableausetup{smalltableaux, boxsize=1.1em} \begin{ytableau}
        \ &\ \\ \none &\none & \ \\ \none &\none &\none &\ &\ \\
        \none &\none &\none &\ &\ &\ \\
         \none &\none &\none &\none \ &\ \\
    \end{ytableau}.\]
\end{example}

\begin{corollary}\label{cor:beta-ptn-SET-min} Let  $\alpha/\beta$  be a \emph{reduced} skew diagram such that $\beta$ is a partition. Then $\SET(\alpha/\beta)$ has a unique minimal element, $S^{\col}_{\alpha/\beta}$.  
Equivalently, the skew extended Hecke poset $\SET(\alpha/\beta)$ coincides with the interval $[S^\col_{\alpha/\beta}, S^\row_{\alpha/\beta}]$, and its   rank is $\binom{|\alpha/\beta|}{2}-\sum_{c \in \alpha/\beta} \mathrm{ls}(c)  $.
\end{corollary}

\begin{proof}  If $\alpha/\beta$ is connected, this follows from \Cref{thm:min_element_connected}. 

Otherwise, we record the following key observation: In the special case when $\beta$ is a partition, the connected components of $\alpha/\beta$ must lie  northwest to southeast along the boundary of $\beta$.  See \Cref{ex:skew-not-conn-not-disjoint-sum}.

This observation guarantees that the straightening algorithm used in the proof of  \Cref{thm:min_element_connected} can also be applied here, even though $\alpha/\beta$ is not connected. Step 2 is the crucial step to verify:  with $i$ defined as in Step 1, all cells with entries less than $i$ and to the right of entry $i$,  must lie in  a row below the row of $i$, 
because of the key observation above and the increase in  rows (left to right) and  columns (bottom to top). In particular $i-1$ lies in a row below the cell containing $i$.

The remaining statements now follow as in the proof of \Cref{thm:min_element_connected}.
\end{proof}

\begin{example}\label{ex:misc-examples}
We record some pertinent examples here.

\begin{enumerate}
\item Let $\alpha=(2,3)$, $\beta=(1,2)$.   Then  $\alpha/\beta$ is not connected, and $S^\col_{\alpha/\beta}$ is not the minimal element.  
We have 
    \[\SET(\alpha/\beta)=\left\{T=\ytableausetup{smalltableaux, boxsize=1.1em} \begin{ytableau}
        \none&1\\2&\none
    \end{ytableau},  \ 
    S^\col_{(2,3)/(1,2)} =S^\text{row}_{(2,3)/(1,2)}=\begin{ytableau}
        \none&2\\1&\none
    \end{ytableau}\right\}.\]
    Here $\pi_1(T)=S^\col_{(2,3)/(1,2)}$ in the skew Hecke poset, 
    showing that  the unique minimal element of $\SET(\alpha/\beta)$ is $T$, and not  $S^\col_{(2,3)/(1,2)}$. 
    This example also shows that if $\SET(\alpha/\beta)$ has a unique minimal element, $\alpha/\beta$ need not be connected.
   
\item    Now consider $\alpha=(3,2)$ and $\beta=(2,1).$
    This example shows that the  converse of \Cref{thm:min_element_connected} 
      is false.  Even if $S^\col_{\alpha/\beta}$ is  the unique minimal element of $\SET(\alpha/\beta)$, $\alpha/\beta$ need not be  connected.  Here 
    $$\SET(\alpha/\beta)=\left\{S^\col_{(3,2)/(2,1)}=\begin{ytableau}
        1&\none\\\none &2
    \end{ytableau}, \quad S^\row_{(3,2)/(2,1)}=\begin{ytableau}
        2&\none\\\none &1
    \end{ytableau}\right\},$$
   with minimal element $S^\col_{(3,2)/(2,1)}$ and maximal element $S^\row_{(3,2)/(2,1)}$. In this case $S^\col_{\alpha/\beta}$ is the unique minimal element of $\SET{(\alpha/\beta)}$, and $\alpha/\beta$ is not connected. 
\end{enumerate}
\end{example}

\section{Disconnected $\alpha/\beta$}\label{sec:disconnected-skew}
\subsection{Disjoint sums} 
\phantom{blank}

We now focus on skew shapes that are disconnected. In the special case when $\alpha/\beta$ is a disjoint sum of components, we have the following. 

\begin{proposition}\label{prop:disj-sum}
Assume that $\alpha/\beta$ is a disjoint sum of   skew shapes 
$(\alpha^{1}/\beta^{1}, \alpha^{2}/\beta^{2},\dots, \alpha^{k}/\beta^{k})$.
Further assume that 
$\alpha^{i}/\beta^{i}$ and $\alpha^{j}/\beta^{j}$ have no common row or column, and $\alpha^{i}/\beta^{i}$ is strictly above $\alpha^{i+1}/\beta^{i+1}$ for $1\le i \le k-1.$ Let $n_i=|\alpha^{i}/\beta^{i}|$, i.e. the number of cells in $\alpha^{i}/\beta^{i}$.
    If $\SET(\alpha^{i}/\beta^{i})$ has a unique minimal element $M_i$ for each $i$, then $\SET(\alpha/\beta)$ has a unique minimal element, obtained by replacing each entry $a$ of $M_i$ with  $a+ n_1+\dots + n_{i-1}.$  
\end{proposition}
\begin{proof}
    By filling the components of $\alpha/\beta$ as described, we create no new inversions between the components, hence keeping the number of inversions minimal.  Since the components do not share a row or column, each column of $\alpha/\beta$ is increasing and so the tableau is an element of $\SET(\alpha/\beta)$.
\end{proof}

An example of the minimal filling of a disjoint sum $\alpha/\beta$ is shown below.  Here for clarity we have not suppressed the cells of $\beta$.
\ytableausetup{smalltableaux, boxsize=1.1em}  
\[\begin{ytableau}
    \,&\,&\,&\,&\,&\,&\,&\,&\,&3&4\\
    \,&\,&\,&\,&\,&\,&\,&\,&1&2\\
    \,&\,&10&12&13\\
    \,&\,&9&11\\
    \,&\,&8\\
    \,&\,&7\\
    \,&5&6\\
    \,&\,&\,&\,&\,&15&19\\
    \,&\,&\,&\,&\,&\,&18\\
    \,&\,&\,&\,&\,&\,&17\\
    \,&\,&\,&\,&\,&14&16\\
\end{ytableau}\]
\subsection{Lobsters}
The right \textit{lobster} $\mathcal{L}^{c_1, c_2}_b$ associated to the triple  of positive integers $(b, c_1, c_2)$ is the skew diagram $\alpha/\beta$ with $\alpha= (b+1+c_2, b+1, b+1+c_1)$ and $\beta = (b+1,1,b+1)$.  Informally, a lobster is a skew diagram with three rows, such that  the top and bottom rows, of respective  lengths $c_1$ and $c_2$, are left-justified on the same vertical line, on which the middle row is right-justified. 
We refer to the middle row as the \emph{body}, and the other two rows as the \emph{claws}.  The claws of a right lobster lie to the \emph{right} of the body.  See Figure \ref{fig:lobster}.

The \emph{left} lobster associated  to the triple  of positive integers $(b, c_1, c_2)$ is defined similarly, with the distinction that now the claws lie to the \emph{left} of the body.

\begin{figure}[htb!]
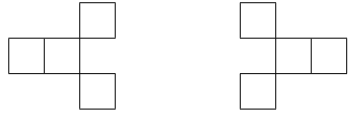

    \centering{
      \young(:::~,:~~,:::~)\qquad  \qquad \young(~,:~~,~)}
    \caption{Right and left lobsters, respectively.}
    \label{fig:lobster} 
\end{figure}

We remark that this class of skew diagrams generalises  the subposet $\SET(\alpha/\beta)$ with three minimal elements in \cite[Figure 1]{LOPSvWwM2025}.  

The total number of elements in the Hecke poset for the lobster $\mathcal{L}_b^{c_1,c_2}$ is 
$|\SET(\mathcal{L}_b^{c_1,c_2})|= \binom{b+c_1+c_2}{b} |\SET(c_2,c_1)|$.  In particular when $c_1=c_2=k$, $|\SET(c_2,c_1)|=C_k$, where $C_k=\frac{1}{k+1}\binom{2k}{k}$ is the $k$th Catalan number, since it is also the number of standard Young tableaux of shape $(k,k)$.  More generally, when $c_1\le c_2$, $|\SET(c_2,c_1)|$ is the number of standard Young tableaux of shape $(c_2, c_1)$ and hence by the hook length formula it equals $\frac{c_2-c_1+1}{c_2+1}\binom{c_1+c_2}{c_1} $.  
It follows that  
\[|\SET(\mathcal{L}_b^{c_1,c_2})|= \frac{c_2-c_1+1}{c_2+1}\binom{b+c_1+c_2}{b, c_1, c_2}, \text{ if $c_1\le c_2$}. \]

We begin with the following special case, for which the skew standard extended Hecke poset has a  well-known structure. \Cref{fig:2fistedlobster} depicts the Hecke poset $\SET(\mathcal{L}_3^{1,1})$. 
\begin{figure}[htb] \centering

\begin{center}{
		\scalebox{.8}{
\begin{tikzpicture}[
    node distance=2.5cm and 2cm,
    arrow/.style={-Latex, thick},
    tpart/.style={draw, inner sep=4pt, align=center, minimum height=2em, minimum width=2em},
    every node/.style={align=center}
]


 \node (m54) at (0,0) { $\begin{ytableau}
     \none&\none&\none&5\\
     1&2&3\\
     \none&\none&\none&4
 \end{ytableau}$ };
\node[below=.4cm of m54] {$M_{54}$};

 \node (m43) at (5,0) { $\begin{ytableau}
     \none&\none&\none&4\\
     1&2&5\\
     \none&\none&\none&3
 \end{ytableau}$ };
\node[below=.4cm of m43] {$M_{43}$};

\node (m32) at (10,0) { $\begin{ytableau}
     \none&\none&\none&3\\
     1&4&5\\
     \none&\none&\none&2
 \end{ytableau}$ };
\node[below=.4cm of m32] {$M_{32}$};

\node (m21) at (15,0) { $\begin{ytableau}
     \none&\none&\none&2\\
     3&4&5\\
     \none&\none&\none&1
 \end{ytableau}$ };
\node[below=.4cm of m21] {$M_{21}$};

\node (r2_1) at (2.5,2.6) { $\begin{ytableau}
     \none&\none&\none&5\\
     1&2&4\\
     \none&\none&\none&3
 \end{ytableau}$ };

\node (r2_2) at (7.5,2.6) { $\begin{ytableau}
     \none&\none&\none&4\\
     1&3&5\\
     \none&\none&\none&2
 \end{ytableau}$ };

\node (r2_3) at (12.5,2.6) { $\begin{ytableau}
     \none&\none&\none&3\\
     2&4&5\\
     \none&\none&\none&1
 \end{ytableau}$ };

\node (r3_1) at (5,5.2) { $\begin{ytableau}
     \none&\none&\none&5\\
     1&3&4\\
     \none&\none&\none&2
 \end{ytableau}$ };

\node (r3_2) at (10,5.2) { $\begin{ytableau}
     \none&\none&\none&4\\
     2&3&5\\
     \none&\none&\none&1
 \end{ytableau}$ };

\node (top) at (7.5,7.8) { $\begin{ytableau}
     \none&\none&\none&5\\
     2&3&4\\
     \none&\none&\none&1
 \end{ytableau}$ };


\draw[arrow] (m54.north) -- (r2_1.south) node[midway, left=3mm] {$\pi_3$};
\draw[arrow] (m43.north) -- (r2_1.south) node[midway, right=1mm] {$\pi_4$};
\draw[arrow] (m43.north) -- (r2_2.south) node[midway, left=3mm] {$\pi_2$};
\draw[arrow] (m32.north) -- (r2_2.south) node[midway, right=1mm] {$\pi_3$};
\draw[arrow] (m32.north) -- (r2_3.south) node[midway, left=3mm] {$\pi_1$};
\draw[arrow] (m21.north) -- (r2_3.south) node[midway, right=1mm] {$\pi_2$};

\draw[arrow] (r2_1.north) -- (r3_1.south) node[midway, left=3mm] {$\pi_2$};
\draw[arrow] (r2_2.north) -- (r3_1.south) node[midway, right=1mm] {$\pi_4$};
\draw[arrow] (r2_2.north) -- (r3_2.south) node[midway, left=3mm] {$\pi_1$};
\draw[arrow] (r2_3.north) -- (r3_2.south) node[midway, right=1mm] {$\pi_3$};

\draw[arrow] (r3_1.north) -- (top.south) node[midway, left=3mm] {$\pi_1$};
\draw[arrow] (r3_2.north) -- (top.south) node[midway, right=1mm] {$\pi_4$};

\end{tikzpicture}
}
}
\end{center}
\caption{The Hecke poset $\SET(\mathcal{L}_3^{1,1})$ with four minimal elements.}
\label{fig:2fistedlobster}
\end{figure}
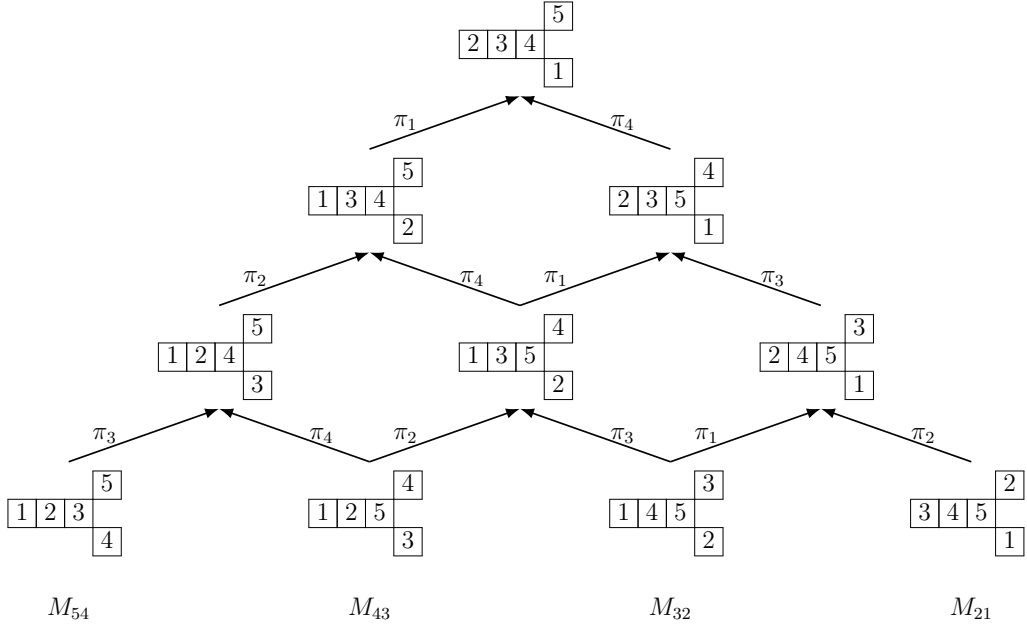
We  explain the structure of the Hecke poset $\SET(\mathcal{L}^{1,1}_{n-2})$ suggested in  \Cref{fig:2fistedlobster}, by showing that it is in fact  isomorphic to the root poset for the root system $A_{n-1}$   \cite[\S 4.6, Figure 4.4]{BjBrenti2005}. A depiction of $A_4$ is shown in Figure \ref{fig:rootposet}. 

In fact, the combinatorial structure of both posets is also related to the following well-known poset \cite[p. 260 and Fig. 3.12]{RPS-EC1-2012}.

Consider the poset  $\mathbb{N}^2$ with maximal element $(0,0)$ and cover relations defined as follows: 
 $(i,j)$ covers $(k,\ell)$ if and only if $k+\ell=i+j+1$ and either $i=k$ or $j=\ell$.
 Thus $(i,j)$ covers exactly two elements, $(i+1,j)$ and $(i, j+1)$.  This makes $\mathbb{N}^2$ an infinite graded poset; it is the dual of the poset described in \cite[p. 260]{RPS-EC1-2012}. Since it is customary to count ranks from the (possibly artificially appended) minimal element of a poset, we use coranks instead. Thus $(0,0)$ has corank 0, and corank $r$ in $\mathbb{N}^2$ consists of the set $\{(i,j): i+j=r\}$. 

\begin{proposition}\label{lobster-2-fisted}  For the right lobster  $\mathcal{L}^{1,1}_{n-2}$ with claws of size 1, the skew standard extended  Hecke poset is isomorphic to 
\begin{itemize}
    \item the root poset for the root system $A_{n-1}$, and also to 
    \item the poset 
 $\mathbb{N}^2$ 
 truncated at corank $n-2$.
\end{itemize}

 It has $n-1$ minimal elements $M_{i+1,i}$, $1\le i\le n-1$, where $M_{i+1,i}$ is the unique filling of the lobster such that  $i, i+1$ appear in the claws.  In particular, all the minimal elements   have the same number of inversions.  The  cardinality of $\SET( \mathcal{L}^{1,1}_{n-2})$ is $\binom{n}{2}$.
\end{proposition}
\begin{proof} The top element of the 0-Hecke poset $\SET(\mathcal{L}^{1,1}_{n-2})$ is  
$ S^{\row}_{\mathcal{L}^{1,1}_{n-2}}= \setlength{\cellsize}{3.8ex}{\tableau{& & & &{\scr n }\\{\scr 2} &{\scr 3 }& \ldots &{\scr n-1}\\ & & & &{\scr 1}}}.$ 

An arbitrary standard extended tableau  $M_{j,i}$ of shape $\mathcal{L}^{1,1}_{n-2}$   is uniquely determined by the pair of entries $i<j$ in the claws, as depicted below, if 
$2\le i<j\le n-1$. 

\[M_{j,i}=  \setlength{\cellsize}{3.8ex}{\tableau{& & & & & & & & & & {\scr j}\\{\scr 1} &{\scr 2 }& \ldots &\boldsymbol{\scr{i-1}} &\boldsymbol{\scr{ i+1}} &\ldots &\boldsymbol{\scr j-1} &\boldsymbol {\scr  j+1}&\ldots &{\scr  n-1}\\ & & & & & & & & & &{\scr i}}}\]
If $i=1$ or $j=n$, the body is filled with the entries $\{2,\ldots,n\}\setminus\{j\}$ and $\{1,\ldots,n-1\}\setminus\{i\}$ respectively.

The definition of the $\rdI$-action now implies that, as long as $j-i\ge 2$, 
we have 
\[M_{j,i}=\pi_i(M_{j,i+1}) = \pi_{j-1}(M_{j-1,i}),
\]
and hence $M_{j,i}$ covers exactly two elements, namely $M_{j,i+1}$ and $M_{j-1,i}$.

Moreover, the elements covering $M_{j,i}$  are $\pi_{i-1}(M_{j,i})=M_{j,i-1}$ if $i\ge 2$ 
and $\pi_j(M_{j,i})=M_{j+1,i}$ if $j\le n-1$.  In particular, $M_{n, i}$, $i\ge 2$ and 
$M_{j,1}$, $j\le n-1$, are each covered by exactly one element, and $M_{j,i}$ is covered by exactly two elements if $j\le n-1$ and $i\ge 2$. 

From Lemma~\ref{lem:key-min}, it follows that the minimal elements are the $M_{j,i}$ with $j-i=1$, i.e., the $n-1$ elements $M_{i+1,i}$. 

One also sees that the Hecke  poset $\SET(\mathcal{L}^{1,1}_{n-2})$ is graded, with corank $n-2$. The elements at corank $r$ (corank 0 is the top element $S^{\row}_{\alpha/\beta}$ where $\alpha/\beta = \mathcal{L}^{1,1}_{n-2}$) are precisely the $r+1$ tableaux $M_{j,i}$ with $j-i=n-r-1$. 
It follows that all the minimal elements are at the same corank, and therefore have the same number of inversions.
  
The description of the $\rdI$-action above shows that the  map  $M_{j,i}\mapsto e_i-e_j $  gives the poset isomorphism with the root poset of $A_{n-1}$.  See \Cref{fig:rootposet}.
\begin{figure}
\begin{center}
{{\scalebox{0.7}{$$\xymatrix @-1.2pc {
& && M_{51} & && \\
&& M_{52} \ar@{-}[ur] & & M_{41} \ar@{-}[ul] && \\
&M_{53} \ar@{-}[ur] & & M_{42} \ar@{-}[ur] \ar@{-}[ul] & & M_{31} \ar@{-}[ul] &\\
M_{54} \ar@{-}[ur] & & M_{43} \ar@{-}[ur] \ar@{-}[ul] & & M_{32}  \ar@{-}[ul] \ar@{-}[ur]&&  M_{21} \ar@{-}[ul]}$$}}}
\qquad
{{\scalebox{0.7}{$$\xymatrix @-1.2pc {
& && e_1-e_5 & && \\
&& e_2-e_5 \ar@{-}[ur] & & e_1-e_4 \ar@{-}[ul] && \\
&e_3 - e_5 \ar@{-}[ur] & & e_2-e_4 \ar@{-}[ur] \ar@{-}[ul] & & e_1-e_3 \ar@{-}[ul] &\\
e_4 - e_5 \ar@{-}[ur] & & e_3-e_4 \ar@{-}[ur] \ar@{-}[ul] & & e_2-e_3  \ar@{-}[ul] \ar@{-}[ur]&&  e_1-e_2 \ar@{-}[ul]}$$}}}
\end{center}
\caption{Isomorphism of $\SET(\mathcal{L}^{1,1}_{n-2})$ with root poset $A_{n-1}$, for $n=5$.}
\label{fig:rootposet}
\end{figure}
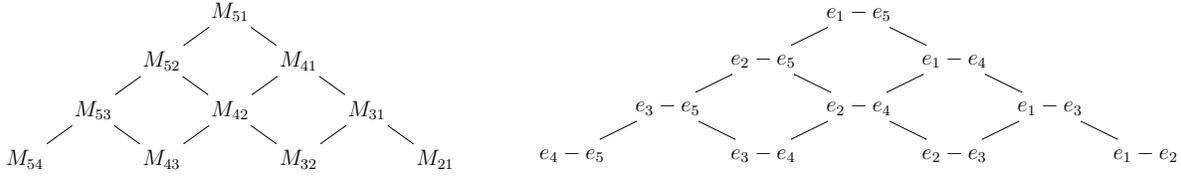

Finally consider the dual of the  poset $\mathbb{N}^2$.  The  map 
\[ M_{j,i}\mapsto (n-j, i-1), 1\le i<j\le n
\]
explicitly establishes the poset isomorphism in the second statement of the theorem.

The statement about the total number of standard extended tableaux follows by adding up elements at each rank, $\sum_{r=0}^{n-2} (r+1)=\binom{n}{2}$.
\end{proof}

Our next observation shows that we need only consider right lobsters.
\begin{lemma}\label{lem:RtoL-lobsters} 
Let $\mathcal{L}^{c_1, c_2}_b$ be the right lobster associated to the triple $(b, c_1, c_2)$.  Let $R$ be the map that sends each entry $i$ of a tableau $T$ in $\SET(\mathcal{L}^{c_1, c_2}_b)$  to $|\alpha/\beta| +1-i$,  and rotates $T$ 
by 180 degrees.  Then $R$ is a poset isomorphism from the skew $\SET$ Hecke poset  of the right lobster $\mathcal{L}^{c_1, c_2}_b$ associated to  the triple $(b, c_1, c_2)$, to the skew $\SET$ Hecke poset of the left lobster associated to the triple $(b, c_2, c_1)$.
\end{lemma}

\begin{proof} First, as an example, 
$R$ maps a tableau whose shape is the  right lobster $\mathcal{L}^{2, 3}_2$, to a tableau whose shape is the left lobster associated to the triple $(2,3,2)$, as follows:
\[
\begin{ytableau}
    \none&\none&\none&3&7\\
    \none&2&6\\
    \none&\none&\none&1&4&5
\end{ytableau}\quad
\raisebox{-10pt}{$\overset{R}{\longmapsto}$}\quad
\begin{ytableau}
    \none&3&4&7\\
    \none&\none&\none&\none&2&6\\
    \none&\none&1&5
\end{ytableau}
\]


Note that if we have $j<j+m$ in a row, then by replacing each entry $i$ by  $|\alpha/\beta| +1-i$ and, rotating by 180 degrees, we get
$$|\alpha/\beta| +1 - j-m <|\alpha/\beta| +1 - j,$$ so after applying $R$ our rows still increase from left to right. Similarly, if we have $j<j+m$ in a column, then by replacing each entry $i$ by  $|\alpha/\beta| +1-i$ and, rotating by 180 degrees, we again get $$|\alpha/\beta| +1 - j-m <|\alpha/\beta| +1 - j,$$ so after applying $R$ our columns still increase from bottom to top. 

It remains to check that, if $\pi _i(T_1) = T_2$ and $T_1\neq T_2$, then
$$\pi_{|\alpha/\beta| +1 - (i+1)}(R(T_1))= R(T_2).$$ 

If $\pi _i(T_1) = T_2$, then this means 
\begin{multicols}{2}
in $T_1$
\begin{itemize}
\item $i$ is above $i+1$, 
\item $i, i+1$ are not in the same column,
\end{itemize} and in $T_2$
\begin{itemize}
\item $i$ is below $i+1$,
\item $i, i+1$ are not in the same column,
\end{itemize}
\end{multicols}
 and all other entries are fixed. 
Observe that
\begin{multicols}{2}
in $R(T_1)$ we have that
\begin{itemize}
\item $|\alpha/\beta| +1 - (i+1)$ is above $|\alpha/\beta| +1 - i$,
\item $|\alpha/\beta| +1 - (i+1), |\alpha/\beta| +1 - i$ are not in the same column,
\end{itemize} and in $R(T_2)$
\begin{itemize}
\item $|\alpha/\beta| +1 - (i+1)$ is below $|\alpha/\beta| +1 - i$,
\item $|\alpha/\beta| +1 - (i+1), |\alpha/\beta| +1 - i$ are not in the same column,
\end{itemize}   
\end{multicols}
and all other entries are fixed. Thus
$$\pi_{|\alpha/\beta| +1 - (i+1)}(R(T_1))= R(T_2),$$ and we are done.
\end{proof}

In view of \Cref{lem:RtoL-lobsters}, henceforth we will consider only right lobsters,  referring to them simply as lobsters.

\begin{theorem}\label{thm:lobster-min-elts} 
    Let $\alpha/\beta$ be the lobster $\mathcal{L}_b^{c_1, c_2}$ with $c= \min(c_1,c_2)$. Then the minimal elements of $\SET(\alpha/\beta)$ are in bijection with binary words having $b$ occurrences of $B$ and $c$ occurrences of $C$. Therefore, $\SET(\alpha/\beta)$ has $\binom{b+c}{b}$ minimal elements, one of which is $S^{\col}_{\alpha/\beta}$. 
\end{theorem}

Before proceeding to the proof, we require one more lemma.

\begin{lemma}\label{lem:columns_min_elements_in_lobsters}
    Let $\alpha/\beta$  be a lobster. Let $T$ be a minimal element of $\SET(\alpha/\beta)$. Then the entries of the cells in any given column of $T$ form an interval of $\mathbb{N}$.
\end{lemma}
\begin{proof}
    We prove the contrapositive. Suppose that there is a column in the claws of $T$ with entries $i$ and $j$, with $j>i+1$. Consider the leftmost such column. Then $j-1$ must be in some other column. It cannot be immediately to the left of $j$, which is in the leftmost column that is not an interval, and $j-2$ cannot be to the left of $i$, as rows are increasing. Also, $j-1$ cannot be to the right of $j$ and in the same row. Thus $j-1$ has to be lower than $j$ and not in the same column. Using Lemma \ref{lem:key-min}, this shows that $T$ is not a minimal of $\SET(\alpha/\beta)$.
\end{proof}
\begin{proof}[Proof of \Cref{thm:lobster-min-elts}]
    Let $T$ be a minimal element in $\SET(\alpha/\beta)$.
    By Lemma \ref{lem:columns_min_elements_in_lobsters}, the entries in each column form an interval of $\mathbb{N}$. Describing $T$ amounts to deciding in what order to fill the columns. We note that, to be an admissible filling of $\SET(\alpha/\beta)$, two columns that have a row in common must be filled from left to right. In a lobster, we therefore have two classes of columns, one corresponding to the body, and one corresponding to the claws.

    The following algorithm to fill the columns of $T$ defines a bijection $\Psi$ from the set of binary words on the alphabet $\{B,C\}$ with $b$ occurrences of $B$ and $c$ occurrences of $C$ to minimal elements of $\SET(\alpha/\beta)=\SET(\mathcal{L}_b^{c_1, c_2})$, with $c= \min(c_1,c_2)$.

    Consider a word $w$ on the binary alphabet $\{B,C\}$, with $b$ occurrences of $B$ and with $c$ occurrences of $C$. We fill the columns of $T = \Psi(w)$ according to the following rules, depicted in Figure \ref{fig:lobster-filling-algorithm}:
    \begin{enumerate}
        \item If $c_1\leq c_2$, for each letter in $w$, write the numbers $\{1,\ldots,n\}$ in order in the cells of $T$, filling the leftmost unfilled column of the body for each occurrence of $B$ and the leftmost unfilled column of the claws for each occurrence of $C$. Then fill the remaining $c_2-c_1$  cells with the numbers $n-(c_2-c_1)+1, \ldots, n$, from left to right.

        To show that these are all minimal elements of $\SET(\alpha/\beta)$, we use Lemma \ref{lem:key-min} and check all the pairs of cells labeled by $i$  and $i+1$. From our construction, they can be such that
        \begin{itemize}
            \item $i$ and $i+1$ are consecutive elements in a row, if both cells are in the body, or if both cells are in the bottom claw with nothing above.
            \item $i$ and $i+1$ are consecutive elements in a column, if both cells are in the claws.
            \item $i$ is at the top of a column. If $i$ is in the upper claw and $i+1$ is not in the same row, then $i+1$ is below. If $i$ is in the body and $i+1$ is not, the filling algorithm forces $i+1$ to be in the lower claw, so $i+1$ is necessarily below $i$.
        \end{itemize}
        In all these cases, the conditions of \Cref{lem:key-min} are satisfied for $T$ to be a minimal element.
        
        Also, we claim that any standard extended tableau that is not filled in this way is not a minimal element of $\SET(\alpha/\beta)$. By Lemma \ref{lem:columns_min_elements_in_lobsters}, each column must  contain an interval of $\mathbb{N}$ to be a minimal element. Hence, the only standard extended tableaux that could be minimal elements and that are not given by the filling algorithms above are those in which at least one of the $n-b-2c$ largest labels is not in  the lower claw. Consider such a tableau and let $i$ be the smallest label that is in one of the last $n-b-2c$ cells of the lower claw and such that $i+1$ is not. Then $i+1$ is in a row above $i$ and not directly above it. By  \Cref{lem:key-min}, such a tableau is not a minimal element of $\SET(\alpha/\beta)$. The above argument shows that $\Psi$ is a bijection from the set of binary words on the alphabet $\{B,C\}$ with $b$ occurrences of $B$ and $c$ occurrences of $C$ to the minimal elements of $\SET(\alpha/\beta)$ when $c_1\leq c_2$.
        \item If $c_1>c_2$, we fill the numbers $\{1,\ldots, n\}$ in order in $\Psi(w)$ according to the letters of $w$ as follows. For each occurrence of $B$, we write the next number in the body, whereas we fill a column of the claw starting at the bottom for each occurrence of $C$. For the last occurrence of $C$, we also fill the remaining cells of the upper claw in order.

        Again, we use \Cref{lem:key-min} to show that these tableaux are exactly the minimal elements of $\SET(\alpha/\beta)$. Let us consider the pairs of cells with labels $(i, i+1)$ in a tableau generated using the above algorithm. Again, they are positioned either such that:
        \begin{itemize}
            \item $i$ and $i+1$ are consecutive elements in a row, if they both appear in the body or if they both appear near the end of the upper claw.
            \item $i$ and $i+1$ are consecutive elements in a column, if they appear in the claws.
            \item $i$ is at the top of a column. Following the algorithm, either $i$ is in the upper claw, and if $i+1$ is not to its right, it must be lower, or $i$ is in the body, and if $i+1$ is not to its right, it is in the lower claw. Therefore, if $i+1$ is neither in the same row nor in the same column of $i$, it must be lower.
        \end{itemize}
        In all these cases, the conditions of \Cref{lem:key-min} are satisfied for $T$ to be a minimal element.

        Next, we need to show that these are all the minimal elements. From  \Cref{lem:columns_min_elements_in_lobsters}, we know that any minimal element has columns filled with intervals on $\mathbb{N}$, meaning that the columns of the claws have consecutive integers. Therefore, the only tableaux that could be minimal elements and are not given by the algorithm above are those whose last entry of the lower claw is not followed by  $n-b-2c+1$ consecutive  entries in the upper claw. In that case, let $j$ be the rightmost entry in the lower claw, and let $j+k$ be the maximal value such that $j+k+1$ is in the upper claw and $j+k$ is not. Necessarily, $j+k$ is lower than $j+k+1$ and not in the same column. From \Cref{lem:key-min}, such a tableau is not a minimal element. Hence, the minimal elements of $\SET(\alpha/\beta)$ are in bijection with the binary words on the alphabet $\{B,C\}$ with $b$ occurrences of $B$ and $c$ occurrences of $C$ through the bijection $\Psi$ when $c_1>c_2$.
    \end{enumerate}
Finally, these arguments also establish the claim that $S^{\col}_{\alpha/\beta}$ is always a minimal element.
\end{proof}

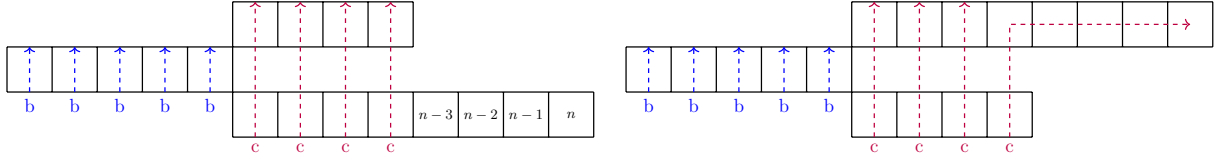
\begin{figure}
    \centering
    \scalebox{0.6}{\begin{tikzpicture}
        \draw[thick] (0,0) grid (8,1);
        \draw[thick] (-5,1) grid (0,2);
        \draw[thick] (0,2) grid (4,3);
        \foreach \x in {0,1,...,4}
          {\draw[blue, dashed, thick, ->] (\x-4.5, 1) node[below] {b} -- (\x-4.5, 2);}
        \foreach \x in {0,1,2,3}
        {\draw[purple, dashed, thick, ->] (\x+0.5, 0) node[below] {c} -- (\x+0.5, 3);}
        \draw (4.5,0.5) node {\scriptsize $n-3$};
        \draw (5.5,0.5) node {\scriptsize $n-2$};
        \draw (6.5,0.5) node {\scriptsize $n-1$};
        \draw (7.5,0.5) node {\scriptsize $n$};
    \end{tikzpicture}}\quad
    \scalebox{0.6}{\begin{tikzpicture}
        \draw[thick] (0,0) grid (4,1);
        \draw[thick] (-5,1) grid (0,2);
        \draw[thick] (0,2) grid (8,3);
        \foreach \x in {0,1,...,4}
          {\draw[blue, dashed, thick, ->] (\x-4.5, 1) node[below] {b} -- (\x-4.5, 2);}
        \foreach \x in {0,1,2}
        {\draw[purple, dashed, thick, ->] (\x+0.5, 0) node[below] {c} -- (\x+0.5, 3);}
        \draw[purple, dashed, thick, ->] (3.5,0) node[below] {c} -- (3.5, 2.5) -- (7.5, 2.5);
    \end{tikzpicture} }
    \caption{The filling algorithms to generate minimal elements in lobsters, an illustration of the bijection $\Psi$.}
    \label{fig:lobster-filling-algorithm}
\end{figure}

The proof of \Cref{thm:lobster-min-elts} defines an explicit bijection $\Psi$ from binary words on the alphabet $\{B,C\}$ with $b$ occurrences of $B$ and $c$ occurrences of $C$ to minimal elements of $\SET(\mathcal{L}^{c_1,c_2}_b)$, where $c=\min(c_1,c_2)$. In this paper, we sometimes refer to the $BC$-word of a minimal element $T$ of $\SET(\alpha/\beta)$ for a lobster $\alpha/\beta$. This indeed refers to the unique binary word $\Psi^{-1}(T)$.

\begin{example}
    Consider the lobster $\mathcal{L}^{3,2}_2$. Here, $c=\min(3,2) = 2$, so there are $6 = \binom{2+2}{2}$ minimal elements, corresponding to the six binary words on the alphabet $\{B,C\}$.
    
    \begin{center}
        \begin{tabular}{c|c|c|c|c|c}
            \young(::467,12,::35) &  \young(::367,14,::25) & \young(::356,17,::24) & \young(::267,34,::15) &  \young(::256,37,::14) & \young(::245,67,::13) \\
            \hline
             BBCC & BCBC & BCCB & CBBC & CBCB & CCBB 
        \end{tabular}
    \end{center}
     Similarly, for the lobster $\mathcal{L}^{2,3}_2$, there are also $6$ minimal elements:
     \begin{center}
        \begin{tabular}{c|c|c|c|c|c}
            \young(::46,12,::357) &  \young(::36,14,::257) & \young(::35,16,::247) & \young(::26,34,::157) &  \young(::25,36,::147) & \young(::24,56,::137) \\
            \hline
             BBCC & BCBC & BCCB & CBBC & CBCB & CCBB 
        \end{tabular}
    \end{center}
\end{example}

We now focus on counting the number of inversions in lobster standard extended tableaux, allowing us to give the rank of the $\SET(\alpha/\beta)$ poset. We start with a specific case, which we then use to find the inversion number of all the minimal elements.
\begin{lemma} \label{lem:inv_S}Let $\alpha/\beta$ be a right lobster with $c_1$, $b$, $c_2$ cells in the top, middle, and bottom row respectively.  Then 
$$\operatorname{inv}(S^\col_{\alpha/\beta})
=c_1(b+\operatorname{min}\{c_1, c_2\})+\binom{b}{2}+\binom{\operatorname{max}\{c_1,c_2\}}{2}.$$
    
\end{lemma}
\begin{proof}
Let $c_1>c_2$. Then

\cellsize=5ex

\[S^\col_{\alpha/\beta}=\tableau{
        & & & & {\scr  b+2}&  {\scr b+4}&\cdots& {\scr b+2c_2}&    \overset{\scr b+2c_2}{\scr +1} &\dots& \overset{\scr b+c_1}{ \scr +c_2}\\
       {\scr 1}  &  {\scr 2}  &\cdots& {\scr b}\\
        &  &   &   & {\scr b+1}  &{\scr b+3}   &\dots&  \overset{ \scr   b+2c_2} {\scr-1}
   }.\]

The number of inversions for each cell is 
\[\phantom{S^\col_{\alpha/\beta}=} \tableau{
        & & & & {\scr  b+1}&  {\scr b+3}&\cdots&  \overset{\scr b+2c_2}{\scr -1}&    {\scr b+2c_2} &\dots& \overset{\scr b+c_1}{ \scr +c_2\!-\!1}\\
       {\scr 0}  &  {\scr 1}  &\cdots& {\scr b-1}\\
        &  &   &   & {\scr 0}  &{\scr 1}   &\dots&   {\scr c_2-1}
   }\]

   and hence \begin{equation*}
       \begin{split}
   \inv(S^{\col}_{\alpha/\beta}) &=\binom{c_2}{2} +\binom{b}{2}+ bc_1 + (1+3+\cdots+(2c_2-1)) + c_2(c_1-c_2) +(c_2+\cdots + (c_1-1)) \\
     &=\binom{c_2}{2} +\binom{b}{2}+ bc_1 +c_2 c_1 + \binom{c_1}{2} -\binom{c_2}{2}\\
     &=c_1(b+c_2) + \binom{b}{2} + \binom{c_1}{2}.
   \end{split}
   \end{equation*}

The analysis for the case $c_1\le c_2$ is similar but more straightforward, so we omit it. The end result is that 
$$\inv(S^\col_{\alpha/\beta}) = 
\begin{cases}
    c_1(b+c_1)+\binom{b}{2}+\binom{c_2}{2}&\text{ if }c_1\le c_2\\
    c_1(b+c_2)+\binom{b}{2}+\binom{c_1}{2}&\text{ if }c_1> c_2
\end{cases},
$$
which reduces to the formula in the statement.
\end{proof}

We are now ready to compute the inversion numbers of all the minimal elements of $\SET(\alpha/\beta).$
\begin{theorem}\label{thm:min_elts}
    Let $\alpha/\beta$ be the right lobster $\mathcal{L}_b^{c_1, c_2}$, with $c_1$, $b$, $c_2$ cells in the top, middle, and bottom row respectively.  Then
    \begin{itemize}
        \item If $c_1\le c_2$, then all minimal elements $T$ of $\SET{(\alpha/\beta)}$ have $\operatorname{inv}(T)=c_1(b+c_1)+\binom{b}{2}+\binom{c_2}{2}.$
        \item If $c_1>c_2$, then a minimal element $T$ has $\operatorname{inv}(T)=c_1(b+c_2)+\binom{b}{2}+\binom{c_1}{2}-(r-1)(c_1-c_2)$, where $r-1$ is the number of $B$'s after the rightmost $C$ in $\Psi^{-1}(T)$, the $BC$-word of $T$, for $1\le r \le b+1$. In particular, there are $\binom{b+c_2-r}{c_2-1}$ minimal elements $T$  with $\operatorname{inv}(T)=c_1(b+c_2)+\binom{b}{2}+\binom{c_1}{2}-(r-1)(c_1-c_2)$ for $1\le r \le b+1$. 
        
       Moreover, in this case, $S^{\col}_{\alpha/\beta}$ is the unique minimal element with the smallest  number of inversions,  among all the minimal elements.     
    \end{itemize}
\end{theorem}
\begin{proof}
    \textbf{Case (1):} Let $c_1\le c_2$. We claim that for minimal elements $T, T'\in \SET(\alpha/\beta)$ such that their corresponding words differ by only a single swap of an adjacent $B$ and $C$ (for example, $B{\color{red}{BC}}BC$ and $B{\color{red}{CB}}BC$), we have $\operatorname{inv}(T)=\operatorname{inv}(T').$  Without loss of generality, assume that $T$ has the substring ``$BC$" and in $T'$ the corresponding substring is ``$CB$".  Then $T$ and $T'$ differ in exactly three cells, which contain consecutive integers:
    \begin{center}
    \ytableausetup{nosmalltableaux, boxsize=2.5em}%
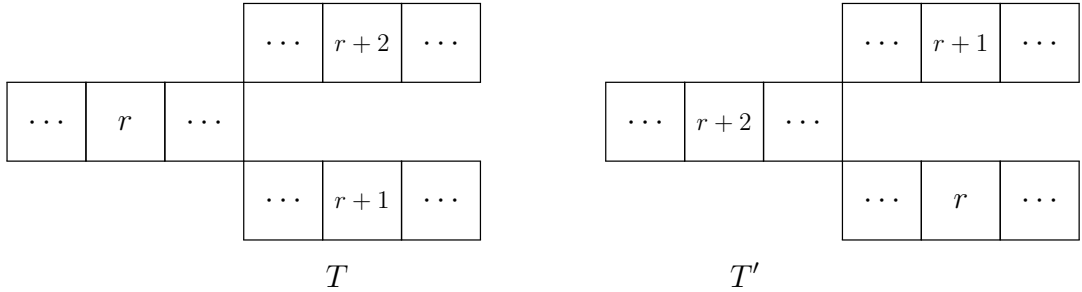
    \begin{ytableau}
         \none&\none&\none&\cdots&\scalebox{0.8}[0.8]{$r+2$}&\cdots\\
        \cdots&r&\cdots\\
        \none&\none&\none&\cdots&\scalebox{0.8}[0.8]{$r+1$}&\cdots\\
    \end{ytableau}\qquad\qquad
    \begin{ytableau}
        \none&\none&\none&\cdots&\scalebox{0.8}[0.8]{$r+1$}&\cdots\\
        \cdots&\scalebox{0.8}[0.8]{$r+2$}&\cdots\\
        \none&\none&\none&\cdots&r&\cdots\\
    \end{ytableau}
    $$T\hskip 2in T'$$

    \end{center}

    Note that $\rw(T)$ and $\rw(T')$ differ only  in the positions of $r$, $r+1$ and $r+2$:
    $$\rw(T) = \cdots r+2\cdots r\cdots r+1\cdots$$
    $$\rw(T') = \cdots r+1\cdots r+2\cdots r\cdots$$
    
    and so for any integer $j\ne r, r+1, r+2$, $j$ is in the same position in both $\rw(T)$ and $\rw(T')$. 
    
    Because $j<r$ or $j>r+2$, the contributions of $j$ to $\operatorname{inv}(T)$ and to $\operatorname{inv}(T')$ are equal.

    It remains to find the difference in the contributions of $r$, $r+1$, and $r+2$ to $\operatorname{inv}(T)$ and $\operatorname{inv}(T')$.  This is equivalent to computing $|\operatorname{inv}(r+2,r,r+1)-\operatorname{inv}(r+1,r+2,r)|$.  But
    $$|\operatorname{inv}(r+2,r,r+1)-\operatorname{inv}(r+1,r+2,r)|=|2-2|=0.$$
    Thus $\operatorname{inv}(T)=\operatorname{inv}(T')$ and swapping an adjacent $B$ and $C$ does not change the inversion number of the corresponding minimal element.
    Since all $BC$-words can be obtained by performing adjacent $BC$-swaps on the $BC$-word for $S^\col_{\alpha/\beta}$, namely the word consisting of $b$ $B$'s, followed by $c_1$ $C$'s, $BB\dots BCC\dots C$, we can now conclude that all minimal elements $T$ of $\SET(\alpha/\beta)$ have 
    $$\operatorname{inv}(T)=\operatorname{inv}(S^\col_{\alpha/\beta})=c_1(b+c_1)+\binom{b}{2}+\binom{c_2}{2},$$
    by Lemma \ref{lem:inv_S}.

    \textbf{Case (2):} Let $c_1>c_2$. This case is exactly the same as \textbf{Case (1)} except when the $BC$-swap involves the rightmost $C$.  Therefore we only need to examine this latter case.
    
    So assume that $T'$ is obtained from $T$ by swapping the rightmost $BC$-pair in the $BC$-word of $T$.  Similar to the proof of \textbf{Case (1)}, we can look at the difference between $T$ and $T'$:

    \ytableausetup{mathmode, boxframe=normal, boxsize=2.5em}
 
   $$\raisebox{-18pt} {$T=$}\quad
   \begin{ytableau}
    \none&\none&\none&\cdots&\scalebox{0.8}[0.8]{$r+2$}&\scalebox{0.8}[0.8]{$r+3$}&\scalebox{0.8}[0.8]{$r+4$}&\cdots&\scalebox{0.9}[0.9]{$\overset{\scr r+c_1}{\scr -c_2+2}$}\\
       \cdots& r &\cdots\\
        \none&\none&\none&\cdots&\scalebox{0.8}[0.8]{$r+1$}
   \end{ytableau}$$
   and 
   $$\raisebox{-18pt} {$T'=$}\quad \begin{ytableau}
       \none&\none&\none&\cdots& \scalebox{0.8}[0.8]{$r+1$}&\scalebox{0.8}[0.8]{$r+2$}&\scr \scalebox{0.8}[0.8]{$r+3$}&\cdots&\scalebox{0.9}[0.9]{$\overset{\scr r+c_1}{\scr -c_2+1}$}\\
       \cdots&\scalebox{0.9}[0.9]{$\overset{\scr r+c_1}{\scr -c_2+2}$}&\cdots\\
        \none&\none&\none&\cdots& r
   \end{ytableau}\, .$$

By the same argument as in \textbf{Case (1)},
      \begin{align*}&\, \operatorname{inv}(T)-\operatorname{inv}(T')\\
   &= \operatorname{inv}({r+c_1-c_2+2},\cdots, r+2,r,r+1)-
    \operatorname{inv}(r+c_1-c_2+1,\cdots, r+1,r+c_1-c_2+2,r).\end{align*}
    
Counting inversions, we have:
$$\operatorname{inv}({r+c_1-c_2+2},\cdots, r+2,r,r+1)=(c_1-c_2+2)+\dots+2+0+0$$
and 
$$\operatorname{inv}(r+c_1-c_2+1,\cdots, r+1,r+c_1-c_2+2,r)=(c_1-c_2+1)+\dots+2+1+1+0.$$
Thus
$$\operatorname{inv}(T)-\operatorname{inv}(T')=(c_1+c_2+2)-2=c_1-c_2,$$
and so swapping the rightmost $BC$-pair in the  $BC$-word of $T$ decreases the inversion number by $c_1-c_2.$

There are $b$ $B$'s that can swap with the rightmost $C$ and so there are $b+1$ possible inversion numbers.  The largest inversion number is that of $S^\col_{\alpha/\beta}$, whose corresponding $BC$-word is $BB\dots B CC\dots C$, so there is no occurrence of $B$ after the rightmost $C$.
The remaining inversion numbers, obtained by swapping $r-1$ occurrences of $B$ with the rightmost $C$, for $1\le r \le b+1,$ are $$\operatorname{inv}(S^\col_{\alpha/\beta})-(r-1)(c_1-c_2)=
c_1(b+c_2)+\binom{b}{2}+\binom{c_1}{2}-(r-1)(c_1-c_2),$$ 
by Lemma \ref{lem:inv_S}. 

The lowest rank (least number of inversions) corresponds to the $BC$-word $CC\dots CBB\dots B$, with $b$ $B$'s after the rightmost $C$, and tableau
\begin{center}
    \ytableausetup
{mathmode, boxframe=normal, boxsize=3em}
\begin{ytableau}
       \none&\none&\none&2&4&\cdots&\scalebox{0.9}[0.9]{$2c_2$}&\cdots&\scalebox{0.8}[0.8]{$c_1+c_2$}\\
       \scalebox{0.6}[0.6]{$c_1+c_2+1$}&\cdots&\scalebox{0.6}[0.6]{$c_1+c_2+b$}\\
        \none&\none&\none&1&3&\cdots&\scalebox{0.8}[0.8]{$2c_2-1$}
   \end{ytableau}\,.
\end{center}
In that case, the rank is \[c_1(b+c_2)+\binom{b}{2}+\binom{c_1}{2}-b(c_1-c_2) = c_2(b+c_1)+\binom{b}{2}+\binom{c_1}{2}.\]

Finally, the number of elements of rank $c_1(b+c_2)+\binom{b}{2}+\binom{c_1}{2}-(r-1)(c_1-c_2)$ equals the number of $BC$-words in which the rightmost $C$ is followed by $r-1$ $B$'s.  Such a word begins with a string of $(c_2-1)$ $C$'s and $(b+1-r)$ $B$'s   in any order, followed by the last $C$ and the remaining $B$'s.  Thus there are 
$\binom{b+c_2-r}{c_2-1}$
such words, as desired.
\end{proof}
\ytableausetup {mathmode, boxframe=normal, boxsize=normal}
\begin{remark} The total number of minimal elements when $c_1>c_2$ equals
$$\sum_{r=1}^{b+1}\binom{b+c_2-r}{c_2-1}= \sum_{m=k}^{b+k}\binom{m}{k}=\binom{b+k+1}{k+1}=\binom{b+c_2}{c_2}=\binom{b+c_2}{b},$$
where we used the substitutions $m=b+c_2-r$ and $k=c_2-1$, as well as the hockey-stick identity.
This agrees with \Cref{thm:lobster-min-elts}.
\end{remark}

We now give the rank of the poset $\SET(\alpha/\beta)$,  i.e., the length of the longest chain in $\SET(\alpha/\beta)$.  Denote this by  $\operatorname{rank}(\SET(\alpha/\beta))$.
\begin{corollary}
    Let $\alpha/\beta$ be the right lobster $\mathcal{L}_b^{c_1,c_2}$ with $c_1$, $b$ and $c_2$ cells in the top, middle and bottom row respectively. Then  
    $$
   \operatorname{rank}(\SET(\alpha/\beta))=
    \begin{cases}
      c_2(b+c_1)-\binom{c_1+1}{2} &\text{ if } c_1\le c_2\\
      bc_1+\binom{c_2}{2} &\text{ if } c_1> c_2\\
    \end{cases}.
    $$
\end{corollary}
\begin{proof}
First, $\operatorname{rank}(\SET(\alpha/\beta))=\operatorname{inv}(S^\row_{\alpha/\beta})-\operatorname{inv}(\text{lowest minimal element}).$
Note that 
$$\operatorname{inv}(S^\row_{\alpha/\beta})
=\binom{b+c_1+c_2}{2}=\binom{b}{2}+\binom{c_1}{2}+\binom{c_2}{2}+bc_1+bc_2+c_1c_2.$$
By Theorem \ref{thm:min_elts},
    $$
    \operatorname{inv}(\text{lowest minimal element})=
    \begin{cases}
        c_1(b+c_1)+\binom{b}{2}+\binom{c_2}{2}&\text{ if } c_1\le c_2\\
        c_2(b+c_1)+\binom{b}{2}+\binom{c_1}{2}&\text{ if } c_1> c_2
    \end{cases}.$$

A calculation now establishes the formulas in the statement, finishing the proof.   
\end{proof}

\noindent
\textbf{Acknowledgments.} This project was initiated in June 2025 at  the   \textit{Collaborative Workshop in Algebraic Combinatorics}, held at the Institute for Advanced Study, Princeton.  The authors gratefully acknowledge the  support of the IAS, as well as partial support from the Natural Sciences and Engineering Research Council of Canada, and NSF grants DMS-
2153998 and DMS-2452044. 

The authors are also very grateful to James Sundstrom for generously offering his invaluable computational expertise. His elegant, powerful and versatile Python code greatly facilitated the process of exploration and discovery in this project.

\bibliographystyle{plain}
\bibliography{PAPER-FINAL-IAS2025Sept6}

\begin{thebibliography}{10}

\bibitem{AS2019}
Sami Assaf and Dominic Searles.
\newblock Kohnert polynomials.
\newblock {\em Experimental Mathematics}, 2019.

\bibitem{BBSSZ2014}
Chris Berg, Nantel Bergeron, Franco Saliola, Luis Serrano, and Mike Zabrocki.
\newblock A lift of the {S}chur and {H}all-{L}ittlewood bases to
  non-commutative symmetric functions.
\newblock {\em Canad. J. Math.}, 66(3):525--565, 2014.

\bibitem{BBSSZ2015}
Chris Berg, Nantel Bergeron, Franco Saliola, Luis Serrano, and Mike Zabrocki.
\newblock Indecomposable modules for the dual immaculate basis of
  quasi-symmetric functions.
\newblock {\em Proc. Amer. Math. Soc.}, 143(3):991--1000, 2015.

\bibitem{BjBrenti2005}
Anders Bj\"{o}rner and Francesco Brenti.
\newblock {\em Combinatorics of {C}oxeter groups}, volume 231 of {\em Graduate
  Texts in Mathematics}.
\newblock Springer, New York, 2005.

\bibitem{HuangJia2016}
Jia Huang.
\newblock A tableau approach to the representation theory of 0-{H}ecke
  algebras.
\newblock {\em Ann. Comb.}, 20(4):831--868, 2016.

\bibitem{LOPSvWwM2025}
Nadia Lafreni\`ere, Rosa Orellana, Anna Pun, Sheila Sundaram, Stephanie van
  Willigenburg, and Tamsen Whitehead~McGinley.
\newblock The skew immaculate {H}ecke poset and 0-{H}ecke modules.
\newblock {\em Electron. J. Combin.}, 32(2):Paper No. 2.11, 34, 2025.

\bibitem{LMvW2013}
Kurt Luoto, Stefan Mykytiuk, and Stephanie van Willigenburg.
\newblock {\em An introduction to quasisymmetric {S}chur functions}.
\newblock SpringerBriefs in Mathematics. Springer, New York, 2013.
\newblock Hopf algebras, quasisymmetric functions, and Young composition
  tableaux.

\bibitem{MathasHecke1999}
Andrew Mathas.
\newblock {\em Iwahori-{H}ecke algebras and {S}chur algebras of the symmetric
  group}, volume~15 of {\em University Lecture Series}.
\newblock American Mathematical Society, Providence, RI, 1999.

\bibitem{NSvWVW2023}
Elizabeth Niese, Sheila Sundaram, Stephanie van Willigenburg, Julianne Vega,
  and Shiyun Wang.
\newblock Row-strict dual immaculate functions.
\newblock {\em Adv. in Appl. Math.}, 149:Paper No. 102540, 33, 2023.

\bibitem{NSvWVW2024}
Elizabeth Niese, Sheila Sundaram, Stephanie van Willigenburg, Julianne Vega,
  and Shiyun Wang.
\newblock 0-{H}ecke modules for row-strict dual immaculate functions.
\newblock {\em Trans. Amer. Math. Soc.}, 377:2525--2582, 2024.

\bibitem{PamelaBromwichNorton1979}
Pamela~N. Norton.
\newblock {$0$}-{H}ecke algebras.
\newblock {\em J. Austral. Math. Soc. Ser. A}, 27(3):337--357, 1979.

\bibitem{RPS-EC1-2012}
Richard~P. Stanley.
\newblock {\em Enumerative combinatorics. {V}olume 1}, volume~49 of {\em
  Cambridge Studies in Advanced Mathematics}.
\newblock Cambridge University Press, Cambridge, second edition, 2012.

\end{thebibliography}

\end{document}